\documentclass[11pt]{article}
\usepackage{latexsym}
\usepackage{amsmath,color}
\usepackage{epsfig}

\newcommand{\vanish}[1]{{}}

\newcommand{\qed}{\mbox{$\Box$}\vspace{\baselineskip}}

\newenvironment{proof}{\noindent {\bf Proof:}}{{\qed}}

\newtheorem{theorem}{Theorem}[section]
\newtheorem{proposition}[theorem]{Proposition}

\newtheorem{definition}[theorem]{Definition}

\makeatletter
\@addtoreset{equation}{section}
\makeatother

\newcommand{\hz}{\hat{0}}
\newcommand{\ho}{\hat{1}}

\newcommand{\ab}{\av\bv}
\newcommand{\av}{{\bf a}}
\newcommand{\bv}{{\bf b}}
\newcommand{\cd}{\cv\dv}

\newcommand{\cv}{{\bf c}}
\newcommand{\dv}{{\bf d}}

\newcommand{\wt}{\operatorname{wt}}

\newcommand{\covered}{\prec}

\begin{document}

\title{On the non-existence of an $R$-labeling\thanks{To appear
in the journal {\em Order.}}}

\author{{\sc Richard EHRENBORG}\thanks{The first author was partially funded by
National Science Foundation grant DMS-0902063.}
          \hspace*{2 mm} and \hspace*{2 mm}
        {\sc Margaret READDY}}

\date{}
\maketitle

\begin{abstract}
We present a family of Eulerian posets which
does not have any $R$-labeling.
The result uses a structure theorem
for $R$-labelings of the butterfly poset.
\end{abstract}

\newcommand{\abc}{\vspace*{-3mm}}
\newcommand{\xyz}{\vspace*{-4mm}}

\abc
\section{Introduction}
\xyz

For a graded poset
the property of having an $R$-labeling is a precursor
to face enumerative results.
The slightly stronger condition of an
$EL$-labeling gives the topological condition of shellability of 
the order complex of the poset.
If a graded poset has an $R$-labeling,
this implies that every entry of its
flag $h$-vector is non-negative.
See~\cite{Bjorner,Stanley_EC_I} for details.
Hence the most straightforward way
to show that a poset lacks an $R$-labeling is to demonstrate a negative
entry in its flag $h$-vector. If the poset has a non-negative
flag $h$-vector, the problem is more difficult.

In this paper
we construct a family of posets where each member has
a positive flag $h$-vector but has no $R$-labeling.
Moreover, half of the examples have
the added attribute that they are Eulerian posets,
that is, each nontrivial interval satisfies
the Euler-Poincar\'e relation.
It is noteworthy that
these Eulerian posets also
have negative coefficients in their $\cd$-indexes.
It is premature for us to assert that 
the lack of an $R$-labeling is related to these negative coefficients.
Further research regarding these types of issues is necessary.

We begin by reviewing the definition of an $R$-labeling,
a notion that has been extended since it was first discovered by
Bj\"orner and Stanley. We reformulate this notion to
a triple assignment.
We then study $R$-labelings of the butterfly poset,
that is,
the unique poset which has two elements of each rank
and every element covers all of the elements of one lower rank.
Using triple assignments we give a structure
theorem for $R$-labelings on the butterfly poset.
We construct a family of  examples by gluing two
butterfly posets together.
The structure theorem is used to show
that these examples cannot have an $R$-labeling.
We end the paper with a number of open questions.

\abc
\section{Graded posets and $R$-labelings}
\xyz

We recall some basic
properties of partially ordered sets (posets),
including
their flag $f$- and flag $h$-vectors. We refer the reader
to Chapter 3 of Stanley's book~\cite{Stanley_EC_I}
for a more complete introduction.
A poset $P$ is {\em graded}
if has a minimal element $\hz$, maximal element $\ho$
and a rank function $\rho$ such that $\rho(\hz) = 0$.
We say that a graded poset $P$ is of rank $n$ if
$\rho(\ho) = n$.
For a poset of rank $n$ and a subset
$S = \{s_{1} < s_{2} < \cdots < s_{k}\}$ of $\{1,2, \ldots, n-1\}$,
define $f_{S}$ to be the number of chains through the ranks of $S$,
that is,
$$   f_{S}
   =
     |\{ \{\hz < x_{1} < x_{2} < \cdots < x_{k} < \ho\}
       \:\: : \:\:
          \rho(x_{i}) = s_{i} \}|   . $$
These $2^{n-1}$ values constitute the {\em flag $f$-vector}
 of the poset.
An equivalent notion is the {\em flag $h$-vector} defined by
the invertible relation
$$   h_{S} = \sum_{T \subseteq S} (-1)^{|S-T|} \cdot f_{T}  . $$

For certain classes of posets
the entries in the flag $h$-vector are non-negative.
This is not at all apparent from the alternating sum
defining the flag $h$-vector.
One explanation of this non-negativity is given
by $R$-labelings.
Let $E(P)$ the set of all cover relations of the poset $P$,
that is,
$E(P)
   =
   \{ (x,y) \in P^{2} \:\: : \:\: x \covered y \}$.
\begin{definition}
An $R$-labeling of a poset $P$ is
a labeling set $\Lambda$ with a relation $\sim$ on its elements
and a function $\lambda: E(P) \longrightarrow \Lambda$
such that in every non-trivial interval $[x,y]$ in the poset $P$
there is a unique maximal chain
$x = x_{0} \covered x_{1} \covered \cdots \covered x_{k} = y$
such that
$\lambda(x_{0},x_{1}) \sim 
 \lambda(x_{1},x_{2}) \sim 
 \cdots \sim
 \lambda(x_{k-1},x_{k})$.
This unique chain is called~{\em rising}.
\label{definition_R}
\end{definition}
In the original definition by Bj\"orner and Stanley~\cite{Bjorner}
the set $\Lambda$ is a totally ordered set.
This was later extended to a partially ordered set $\Lambda$
by Bj\"orner and Wachs~\cite{Bjorner_Wachs}.
However, since none of the poset axioms are used from the
poset $\Lambda$, the most general definition so far is the
one given above.

The next result presents the connection between
$R$-labelings and the flag $h$-vector.
For a maximal chain
$c = \{\hz = x_{0} \covered x_{1} \covered x_{2} \covered
     \cdots \covered x_{n} = \ho\}$, we define its {\em descent
set} to be
$$   D(c)
   =
     \{i \in \{1, \ldots, n-1\} \:\: : \:\: 
            \lambda(x_{i-1},x_{i}) \not\sim \lambda(x_{i},x_{i+1})\} . $$
\begin{theorem}[Bj\"orner and Stanley]
Let $P$ be a graded poset with an $R$-labeling.
The number of maximal chains with descent set $S$ is given by
the flag $h$-vector entry $h_{S}$.
\end{theorem}
Although we extended the original notion of $R$-labelings,
the proof in~\cite{Bjorner} still applies.

For a poset $P$ let $W(P)$ denote the set
of triplets of elements that cover each other,
that is,
$$
   W(P) 
           =
   \{ (x,y,z) \in P^{3} \:\: : \:\: x \covered y \covered z \} .
$$

\begin{definition}
A {\em triple assignment} of a poset is a function
$\tau: W(P) \longrightarrow \{\av,\bv\}$
such that for every non-trivial interval $[x,y]$ in the poset
there is a unique maximal chain
$x = x_{0} \covered x_{1} \covered x_{2} \covered \cdots \covered x_{k} = y$
such that $\tau(x_{i},x_{i+1},x_{i+2}) = \av$ for $0 \leq i \leq k-2$.
\end{definition}

\begin{proposition}
The two notions of $R$-labelings and triple assignments are equivalent.
\end{proposition}
\begin{proof}
Given a labeling $\lambda$ of the poset $P$,
define the function $\tau$ by
$\tau(x,y,z) = \av$ if and only if $\lambda(x,y) \sim \lambda(y,z)$.
If $\lambda$ is an $R$-labeling then directly we have that
$\tau$ is a triple assignment.

Conversely, let  $\tau : W(P) \longrightarrow \{\av,\bv\}$
be a triple assignment function.
We define a labeling as follows.
Let the label
set $\Lambda$ be the set of all cover relations, that is,
$\Lambda = E(P)$
and the labeling $\lambda$ is given by
$\lambda(x,y) = (x,y)$.
Define the relation $\sim$ on~$\Lambda$ by
$(x,y) \sim (y,z)$ if and only if $\tau(x,y,z) = \av$.
It follows now that if $\tau$ is a triple assignment
then the labeling $\lambda$ is an $R$-labeling.
\end{proof}

The reason the two element set $\{\av,\bv\}$ is used as the
range of a triple function stems from the notion of the
$\ab$-index of a poset.
Let $\av$ and $\bv$ be two non-commutative variables of degree $1$.
For $S$ a subset of the set $\{1, 2, \ldots, n-1\}$
define the monomial $u_{S} = u_{1} u_{2} \cdots u_{n-1}$
by letting $u_{i} = \bv$ if $i \in S$ and
otherwise $u_{i} = \av$. The {\em $\ab$-index} is the
non-commutative polynomial 
$$   \Psi(P)
   =
     \sum_{S}
            h_{S} \cdot u_{S}   .  $$
The $\ab$-index is an equivalent encoding of the flag $h$-vector
of a poset and it has degree one less than the rank of the poset.

For a maximal chain
$c = \{\hz = x_{0} \covered x_{1} \covered x_{2} 
                   \covered \cdots \covered x_{n} = \ho\}$,
define its {\em weight} $\wt(c)$ by
the product
$\wt(c)
  =
      \tau(x_{0},x_{1},x_{2})
    \cdot
      \tau(x_{1},x_{2},x_{3})
    \cdots
      \tau(x_{n-2},x_{n-1},x_{n})$.
The $\ab$-index of a poset $P$
having triple assignment $\tau$
is then given by
$\Psi(P) = \sum_{c} \wt(c)$,
where the sum is over all maximal chains $c$ in~$P$.

Recall a poset is {\em Eulerian} if every non-trivial interval satisfies
the Euler-Poincar\'e relation, that is,
it has the same number of elements of
odd rank as even rank. For Eulerian posets Bayer and
Klapper~\cite{Bayer_Klapper} proved
that the $\ab$-index can be written in terms of
the non-commutative variables $\cv = \av + \bv$
and $\dv = \av\bv + \bv\av$.
This invariant is called the {\em $\cd$-index}.
It offers an efficient encoding of
the entries of the flag $h$-vector of an Eulerian poset.
That a poset has a $\cd$-index is equivalent to that
the flag $f$-vector of the poset satisfies the
generalized Dehn-Somerville relations;
see~\cite{Bayer_Billera}.

\abc
\section{The butterfly poset}
\xyz

The butterfly poset $T_{n}$ is the unique graded poset
of rank $n$ such that there are two elements of rank~$i$
for $1 \leq i \leq n-1$ and every element
different from $\hz$
covers all of the elements
of one rank below.
Note that every interval in the butterfly poset
is a butterfly poset of smaller rank
and that the butterfly poset is an Eulerian poset.
We will denote the elements
of $T_{n}$ by
$\{\hz, x_{1}, \overline{x_{1}}, \ldots, x_{n-1}, \overline{x_{n-1}}, \ho\}$,
where $\rho(x_{i}) = \rho(\overline{x_{i}}) = i$.
For an element $x$ in the butterfly poset
different from the minimal and maximal elements,
let $\overline{x}$ denote the unique element different from $x$
but of the same rank as $x$.
Furthermore, let $\overline{\: \cdot \:}$ also denote the natural involution
on the $2$-element set $\{\av,\bv\}$, that is,
$\overline{\av} = \bv$ and $\overline{\bv} = \av$.

It is easy to verify that the flag $f$- and flag $h$-vectors
of the butterfly poset $T_{n}$ are given by
$$      f_{S} = 2^{|S|} 
    \:\:\:\: \mbox{ and } \:\:\:\:
        h_{S} = 1      $$
for $S$ a subset of the set $\{1,2, \ldots, n-1\}$.
Hence the $\ab$-index of the butterfly poset
is given by
$\Psi(T_n) = (\av+\bv)^{n-1} = \cv^{n-1}$.

Assume $\tau$ is a function
$\tau : W(T_{n}) \longrightarrow \{\av,\bv\}$
such that every interval of length $2$ has a unique
rising chain. Since every length $2$ interval $[x,z]$
is a diamond, we conclude that
$\tau(x,\overline{y},z) = \overline{\tau(x,y,z)}$
where $y$ and $\overline{y}$ are the two unique atoms
(and coatoms!) in the interval $[x,z]$.

For a function
$\tau: W(T_{n}) \longrightarrow \{\av,\bv\}$
call an element $y$ a {\em breakpoint} if
$\hz < y < \ho$ and the value
of $\tau(x,y,z)$ does not depend on $x$ and $z$.
Note that if $y$ is a breakpoint then so is
$\overline{y}$.

\begin{theorem}
Let $n$ be a positive integer greater than or equal
to $2$ and let $\tau$ be a function
$\tau: W(T_{n}) \longrightarrow \{\av,\bv\}$
such that
in every interval of rank $3$ or less there
is a unique rising chain. Then the following two
statements hold:
\vspace*{-5mm}
\begin{itemize}
\item[(i)]
There is a breakpoint $y$ in the poset $T_{n}$.
\vspace*{-2mm}
\item[(ii)]
The function $\tau$ is a triple assignment.
\end{itemize}
\vspace*{-5mm}
\label{theorem_butterfly}
\end{theorem}
\begin{proof}
First we show the existence of a breakpoint
by induction on the rank.
For the case $n=2$ the statement is straightforward
to verify.
Next consider the case $n=3$. 
Assuming that $\overline{x_{1}}$ is not a breakpoint,
we have that
$\tau(\hz,\overline{x_{1}},\overline{x_{2}})
  =
 \overline{\tau(\hz,\overline{x_{1}},x_{2})}
  =
 \tau(\hz,x_{1},x_{2})$.
Similarly, assuming that $\overline{x_{2}}$
is not a breakpoint we have 
$\tau(\overline{x_{1}},\overline{x_{2}},\ho)
  =
 \overline{\tau(x_{1},\overline{x_{2}},\ho)}
  =
 \tau(x_{1},x_{2},\ho)$.
Hence the two chains
$\{\hz < x_{1} < x_{2} < \ho\}$
and
$\{\hz < \overline{x_{1}} < \overline{x_{2}} < \ho\}$
have the same weight.
This is a contradiction
since every entry in the flag $h$-vector of $T_{3}$ is
at most $1$.
Thus at least one assumption is wrong and
we conclude that there is a breakpoint.

For the induction step, assume that $n \geq 4$.
Consider the three intervals
$[x_{1},\ho]$,
$[\overline{x_{1}},\ho]$
and
$[\hz, x_{3}]$.
All are butterfly posets of rank less than $n$
and the induction hypothesis holds for them.
Hence the interval $[x_{1},\ho]$
contains a breakpoint $x_{i}$, for some $2 \leq i \leq n-1$.
If $i \geq 3$ this is a breakpoint for the whole
poset and we are done. Hence we assume that $i=2$
and we have that
$\tau(x_{1},x_{2},x_{3}) = \tau(x_{1},x_{2},\overline{x_{3}})$.
Similarly,
the interval
$[\overline{x_{1}},\ho]$ has a breakpoint.
Avoiding a breakpoint of rank $3$ or higher
in $[\overline{x_{1}},\ho]$,
yields
$\tau(\overline{x_{1}},x_{2},x_{3})
 = \tau(\overline{x_{1}},x_{2},\overline{x_{3}})$.

Finally, the interval
$[\hz,x_{3}]$ contains a breakpoint.
If it is $x_{1}$ then it is a breakpoint for the
entire poset. If it is $x_{2}$ we have
that
$\tau(x_{1},x_{2},x_{3}) = \tau(\overline{x_{1}},x_{2},x_{3})$.
By concatenating these three equalities we obtain that $x_{2}$
is a breakpoint for the poset $T_{n}$,
completing the induction.

It remains to show that $\tau$ is
a triple assignment.
Let $[x,y]$ be an interval in $T_{n}$.
Since the interval $[x,y]$ is isomorphic to
a butterfly poset there
is a breakpoint $z$ in this interval.
Also note that $\overline{z}$ is also a breakpoint.
Without loss of generality we may assume that
the value of the function $\tau$ at $z$ is $\av$,
that is, 
$\tau(z^{\prime},z,z^{\prime\prime}) = \av$
for all $z^{\prime}$ and $z^{\prime\prime}$.
Now concatenate the two unique rising chains in
the intervals $[x,z]$ and $[z,y]$.
The result is a rising chain.
Furthermore,
it is the only possible rising chain in the interval $[x,y]$.
This proves that $\tau$ is a triple assignment.
\end{proof}

\begin{figure}
\setlength{\unitlength}{1.0mm}
\newcommand{\point}{\circle*{1.5}}

\begin{center}
\begin{picture}(30,30)(0,0)

\put(15,0){\point}
\multiput(0,10)(0,10){2}{\multiput(0,0)(10,0){4}{\put(0,0){\point}}}
\put(15,30){\point}

\put(15,0){\line(-3,2){15}}
\put(15,0){\line(-1,2){5}}
\put(15,0){\line(1,2){5}}
\put(15,0){\line(3,2){15}}

\multiput(0,10)(10,0){4}{\put(0,0){\line(0,1){10}}}
\multiput(0,10)(20,0){2}{\put(0,0){\line(1,1){10}}}
\multiput(10,10)(20,0){2}{\put(0,0){\line(-1,1){10}}}

\put(15,30){\line(-3,-2){15}}
\put(15,30){\line(-1,-2){5}}
\put(15,30){\line(1,-2){5}}
\put(15,30){\line(3,-2){15}}

\put(5,3.5){{\scriptsize 1}}
\put(11,3.5){{\scriptsize 3}}
\put(18,3.5){{\scriptsize 2}}
\put(24.5,3.5){{\scriptsize 2}}

\put(-2,14){{\scriptsize 2}}
\put(1,13){{\scriptsize 2}}
\put(7,13){{\scriptsize 2}}
\put(11,14){{\scriptsize 2}}

\put(18,14){{\scriptsize 1}}
\put(21,13){{\scriptsize 3}}
\put(27,13){{\scriptsize 3}}
\put(31,14){{\scriptsize 1}}

\put(5,24.5){{\scriptsize 3}}
\put(11,24.5){{\scriptsize 1}}
\put(18,24.5){{\scriptsize 2}}
\put(24.5,24.5){{\scriptsize 2}}

\end{picture}
\hspace*{30 mm}
\begin{picture}(30,50)(0,0)

\put(15,0){\point}
\multiput(0,10)(0,10){4}{\multiput(0,0)(10,0){4}{\put(0,0){\point}}}
\put(15,50){\point}

\put(15,0){\line(-3,2){15}}
\put(15,0){\line(-1,2){5}}
\put(15,0){\line(1,2){5}}
\put(15,0){\line(3,2){15}}

\multiput(0,0)(0,10){3}{
  \multiput(0,10)(10,0){4}{\put(0,0){\line(0,1){10}}}
  \multiput(0,10)(20,0){2}{\put(0,0){\line(1,1){10}}}
  \multiput(10,10)(20,0){2}{\put(0,0){\line(-1,1){10}}}}

\put(15,50){\line(-3,-2){15}}
\put(15,50){\line(-1,-2){5}}
\put(15,50){\line(1,-2){5}}
\put(15,50){\line(3,-2){15}}

\end{picture}
\end{center}
\caption{The two Eulerian posets $P_{3}$ and $P_{5}$.
The poset $P_{3}$ has an $R$-labeling, whereas $P_{5}$
does not.}
\label{figure_one}
\end{figure}
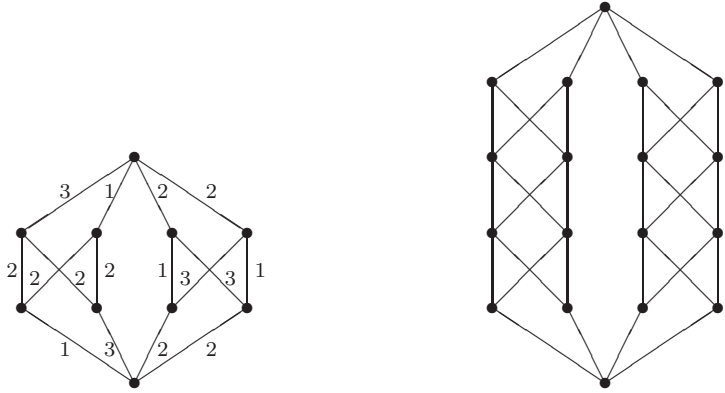

\abc
\section{A class of posets without an $R$-labeling}
\xyz

Let $P_{n}$ consist of two copies of the butterfly poset $T_{n}$
where we have identified the minimal elements and the maximal
elements. See Figure~\ref{figure_one} for the two posets
$P_{3}$ and $P_{5}$.

The flag $f$-vector of the poset $P_{n}$ is given by
$f_{S} = 2 \cdot 2^{|S|}$ for $S$ non-empty and
$f_{\emptyset} = 1$.
Hence its flag $h$-vector is non-negative and is given by
$$  h_{S}(P_{n})
       =  2 - (-1)^{|S|}
       =       \left\{  \begin{array}{c l}
                         1 & \mbox{ if } |S| \mbox{ is even,} \\
                         3 & \mbox{ if } |S| \mbox{ is odd.}
                         \end{array} \right. $$
Another way to observe this is to compute
the $\ab$-index of this poset. It is
$\Psi(P_{n}) = 2 \cdot \Psi(T_{n}) - (\av-\bv)^{n-1}
             = 2 \cdot \cv^{n-1} - (\av-\bv)^{n-1}$;
see for instance~\cite[Section~11]{Ehrenborg_Readdy_c}.

When $n$ is odd the poset $P_{n}$ is Eulerian.
In fact, its $\cd$-index is given by
$\Psi(P_{n}) = 2 \cdot \cv^{2k} - (\cv^{2} - 2 \cdot \dv)^{k}$
for $n = 2k+1$. For $k \geq 2$ note that 
every $\cd$-monomial having an even number of $\dv$'s and different
from the monomial~$\cv^{n-1}$ has a negative coefficient.

\begin{theorem}
The poset $P_{n}$ for $n \geq 4$ does not have an $R$-labeling.
\end{theorem}
\begin{proof}
Let $P$ and $Q$ be the two subposets of $P_{n}$
such that they are both isomorphic to $T_{n}$,
their union is $P_{n}$ and they intersect in $\{\hz,\ho\}$.
Assume that $P_{n}$ has a triple assignment $\tau$.
Consider $\tau$ restricted to the subposet $P$.
Since every interval of
length $3$ or less in $P$ is an interval in $P_{n}$,
the poset $P$ with the function $\tau$ satisfies the
condition of Theorem~\ref{theorem_butterfly}.
Hence $\tau$ is a triple assignment for the poset $P$.
Hence there is a rising chain in the poset $P$.

By the exact same reasoning,
there is a rising chain in the poset $Q$,
yielding the contradiction that~$P_{n}$
has two rising chains.
\end{proof}

\abc
\section{Concluding remarks}
\xyz

In the literature there are examples of
non-shellable
simplicial complexes whose geometric realization are
$3$-dimensional balls and spheres.
For instance, see~\cite{Ehrenborg_Hachimori,Lickorish,Rudin,Vince}
and the references therein.
The difficulty in each of these papers is not to find a non-shellable
complex, but to find a non-shellable object
having a natural geometric realization.
Thus we sharpen the question of this paper to:
Is there a poset which lacks an $R$-labeling having
a positive flag $h$-vector such that
\vspace*{-6mm}
\begin{itemize}
\item[(i)]
it is also a lattice?
\vspace*{-3mm}
\item[(ii)]
its chain complex has the geometric realization of a sphere?
\end{itemize}
\vspace*{-6mm}
Furthermore, can one find a poset having $R$-labelings
but where none of the labelings is an $EL$-labeling? 
The similar question concerning whether there are posets
which are shellable
but not $EL$-shellable has been answered independently in
two papers~\cite{Vince_Wachs,Walker}.

Observe that the poset $P_{n}$ for $n$ odd is obtained by
doubling a half-Eulerian poset. This notion was introduced by Bayer
and Hetyei~\cite{Bayer_Hetyei}.
Their paper gives a plethora of examples.
What is known about labelings for these posets in general?

\abc
\section*{Acknowledgements}
\xyz

The authors would like to thank the referee
for his comments.

\newcommand{\journal}[6]{{\sc #1,} #2, {\it #3} {\bf #4} (#5), #6.}
\newcommand{\book}[4]{{\sc #1,} ``#2,'' #3, #4.}
\newcommand{\bookf}[5]{{\sc #1,} ``#2,'' #3, #4, #5.}
\newcommand{\books}[6]{{\sc #1,} ``#2,'' #3, #4, #5, #6.}
\newcommand{\collection}[6]{{\sc #1,}  #2, #3, in {\it #4}, #5, #6.}
\newcommand{\thesis}[4]{{\sc #1,} ``#2,'' Doctoral dissertation, #3, #4.}
\newcommand{\springer}[4]{{\sc #1,} ``#2,'' Lecture Notes in Math.,
                          Vol.\ #3, Springer-Verlag, Berlin, #4.}
\newcommand{\preprint}[3]{{\sc #1,} #2, preprint #3.}
\newcommand{\preparation}[2]{{\sc #1,} #2, in preparation.}
\newcommand{\appear}[3]{{\sc #1,} #2, to appear in {\it #3}}
\newcommand{\submitted}[3]{{\sc #1,} #2, submitted to {\it #3}}
\newcommand{\JCTA}{J.\ Combin.\ Theory Ser.\ A}
\newcommand{\AdvancesinMathematics}{Adv.\ Math.}
\newcommand{\JournalofAlgebraicCombinatorics}{J.\ Algebraic Combin.}

\newcommand{\communication}[1]{{\sc #1,} personal communication.}


\bigskip

\noindent
{\em R.\ Ehrenborg and M.\ Readdy,
Department of Mathematics,
University of Kentucky,
Lexington, KY 40506-0027,}
\{{\tt jrge},{\tt readdy}\}{\tt @ms.uky.edu}

\end{document}